\documentclass{ws-jaa}
\DeclareMathOperator*{\dm}{\mathrm dim}

\DeclareMathOperator*{\vr}{\mathrm Var}
\DeclareMathOperator*{\oprR}{\mathit R}
\DeclareMathOperator*{\oprH}{\mathit H}
\DeclareMathOperator*{\oprT}{\mathit T}
\newcommand{\ass}[3]{\left(#1,#2,#3\right)}
\newcommand{\dva}[2]{\left(#1#2\right)}
\newcommand{\com}[2]{\left[#1,#2\right]}
\newcommand{\rl}[1]{\mathrm{RA}^{\left<#1\right>}_{2}}
\newcommand{\slch}[2]{\genfrac{[}{]}{}{}{#1}{#2}}
\newcommand{\jord}[2]{#1\circ #2}
\newcommand{\tri}[3]{\bigl(\left(#1#2\right)#3\bigr)}
\newcommand{\jos}[2]{\left(#1\circ #2\right)}

\newcommand{\md}[1]{\left|#1\right|}
\newcommand{\Gob}[1]{{\mathrm G}\left(#1\right)}
\newcommand{\svob}[1]{F_{\scriptstyle #1}\left[X\right]}
\newlength{\fchr}
\newcommand{\factor}[2]{
\lefteqn{\phantom{#1}\hspace{-0.5\fchr}\diagup}
\genfrac{}{}{0pt}{0}{{#1}\phantom{#2}} {\phantom{#1}#2}}
\begin{document}

\markboth{Alexey Kuz'min}
{On the topological rank of the variety of right alternative metabelian Lie-nilpotent algebras}

\catchline{}{}{}{}{}
%
%%%%%%%%%%%%%%%%%%%%%%%%%%%%%%%%%%%%%%%%%%%%%%%%%%%%%%%%%%%%%%%%%%%%

\title{ON THE TOPOLOGICAL RANK OF THE VARIETY OF RIGHT ALTERNATIVE METABELIAN LIE-NILPOTENT ALGEBRAS}

\author{\footnotesize ALEXEY KUZ'MIN\footnote{The author is supported by the S\~ao Paulo Research Foundation (FAPESP) 2010/51880--2.}}

\address{Institute of Math. and Stat.,
University of S\~ao Paulo\\
Rua do Matao, 1010 -- Cidade Universitaria\\
S\~ao Paulo, S\~ao Paulo, 05508-090, Brazil\\
\email{amkuzmin@ya.ru}}

\maketitle

\begin{history}
\received{(Day Month Year)}
\revised{(Day Month Year)}
\accepted{(Day Month Year)}
\comby{(xxxxxxxxx)}
\end{history}

\begin{abstract}
In 1981,
S.~V.~Pchelintsev  introduced the notion of topological rank for Spechtian varieties of algebras
as a certain tool for studying the structure of non-nilpotent subvarieties in a given variety.
We provide a variety of right alternative algebras of arbitrary given finite topological rank.
Namely, we prove that the topological rank of the variety of right alternative metabelian (solvable of index two)
algebras that are Lie-nilpotent of step not more than~$s$
over a field of characteristic distinct from two and three
is equal to~$s$.
\end{abstract}



\section{Introduction}
In 1986, A.~R.~Kemer~\cite{Kemer87,Kemer88} solved affirmatively the famous Specht
problem~\cite{Specht50} by proving that an arbitrary variety
of associative algebras over a field of characteristic zero is finitely based.
It is also known~\cite{Grishin99}--\cite{Schigolev00} that
there are non-finitely based varieties of associative algebras
over an arbitrary field of prime characteristic.

Recall that a variety of algebras is said to be
\textit{Spechtian}
(or to have the \textit{Specht property}) if its every subvariety is finitely based.
The Kemer's theorem has certain analogs in the cases of Jordan, alternative, and Lie algebras
over a field of characteristic~$0$.
Namely, A.~Ya.~Vais, E.~I.~Zel'manov~\cite{Vais-Zelmanov89}
proved the Specht property of
the variety generated by Jordan
$\mathrm{PI}\text{-algebra}$
on a finite set of generators.
A.~V.~Iltyakov~\cite{Iltyakov91} obtained the similar result for alternative
$\mathrm{PI}\text{-algebras}$ and also proved in~\cite{Iltyakov92} that the variety generated by
a finite dimensional Lie algebra is Spechtian.
U.~U.~Umirbaev~\cite{Umirbaev85}
proved the Specht property of every variety of solvable alternative algebras
over a field of characteristic distinct from two and three.
The questions about the Specht property for the varieties of all alternative, Lie, and Jordan algebras
over a field of characteristic zero are still open problems.

Since 1976, it is known~\cite{Belkin76} that the variety
of all right alternative metabelian algebras over an arbitrary field is not Spechtian.
In 1985,
I.~M.~Isaev~\cite{Isaev86} proved that non-finitely based varieties of
right alternative metabelian algebras can even be generated by finite-dimensional algebras.
The Specht property for certain varieties of right alternative algebras
were also studied in~\cite{Medvedev78}--\cite{Pchelintsev07}.

Recall~\cite{Pchelintsev81} the notion of topological rank for Spechtian varieties of algebras.
Let~$\mathcal V$ be a Spechtian variety of algebras and
$\mathfrak M$
be a proper subvariety
of
$\mathcal V$.
By a
\textit{system distinguishing $\mathfrak M$ from $\mathcal V$}
we mean a set
$S=\left\{f_1,\ldots,\,f_n\right\}$
of nonzero homogeneous polynomials of the free
$\mathcal V$-algebra
such that
$\mathfrak M$
can be defined by the union of the identities
$f_1=0,\ldots,\,f_n=0$
with the defining identities of $\mathcal V$.
A
\textit{degree of the system $S$}
is the maximal degree of its polynomials.
A
\textit{dimension
$\dm_{\mathcal V}\mathfrak M$
of~$\mathfrak M$
with respect to
$\mathcal V$}
is the minimal possible degree of a system distinguishing
$\mathfrak M$ from~$\mathcal V$.
Let
$\wp(\mathcal V)$
be the
\textit{set
of all subvarieties of~$\mathcal V$}.
For every
$\mathfrak R\in\wp(\mathcal V)$
and
$n\in\mathbb{N}$
by
$\textstyle\stackrel{\circ}{\mathrm U}_n(\mathfrak R)$
we denote the
\textit{set of all proper subvarieties
$\mathfrak M\subset\mathfrak R$
of dimensions
$\dm_{\mathfrak R}\mathfrak M\geqslant n$}
and put
$\mathrm{U}_n (\mathfrak R)
=
\stackrel{\circ}{\mathrm U}_n(\mathfrak R)\cup\left\{\mathfrak R\right\}$.
By definition, it is clear that
$$
\mathrm{U}_n (\mathfrak R)\cap
\mathrm{U}_{n'} (\mathfrak R)=
\mathrm{U}_{\max{(n,n')}} (\mathfrak R)
$$
and, for every
$
\mathfrak M\in\textstyle\stackrel{\circ}{\mathrm U}_n(\mathfrak R)
$,
we have
$
\mathrm{U}_{n} (\mathfrak M)\subset
\mathrm{U}_{n} (\mathfrak R)
$.
Therefore considering a set
$$
\mathfrak B =\left\{\mathrm{U}_n (\mathfrak R)\mid
\mathfrak R\in\wp(\mathcal V),\;n\in\mathbb{N}\right\}
$$
as a base for the neighborhoods, we endow
$\wp(\mathcal V)$
with a topology.
Every subset
$\wp(\mathcal V)$
gains a structure of the topological space with respect to the induced topology.
For every
$\Omega\subseteq\wp(\mathcal V)$
by
$\Omega'$
denote a
\textit{subspace of
$\Omega$
obtained by the exclusion of all its isolated points.}
Note that by virtue of the Specht property of
$\mathcal V$,
every descending chain of varieties in
$\wp(\mathcal V)$
stabilizes and, consequently, every topological subspace of
$\wp(\mathcal V)$
contains isolated points.
Therefore one can consider the strictly descending chain
$$
\Omega\supset\Omega'\supset\Omega''\supset\cdots\supset\Omega^{(n)}\supset\cdots.
$$
The
\textit{topological rank $\mathrm{r_t}(\Omega)$ of the space $\Omega$}
is the minimal $n$ such that
$\Omega^{(n)}=\varnothing$ or
$\aleph_0$ if such an $n$ doesn't exist.
The value
$\mathrm{r_t}\bigl(\wp(\mathcal V)\bigr)$
is called the
\textit{topological rank of the variety $\mathcal V$}
and is denoted shortly by
$\mathrm{r_t}(\mathcal V)$.
For example, if
$\mathcal V$ is nilpotent, then every subvariety of
$\mathcal V$
turns out to be an isolated point of
$\Omega=\wp(\mathcal V)$ and, consequently,
$\mathrm{r_t}(\mathcal V)=1$.
Otherwise,
$\Omega'$
consists of all non-nilpotent subvarieties of~$\mathcal V$
and its isolated points are the varieties that were limit points in
$\Omega$
for the sequences of nilpotent varieties only.
Further, if
$\Omega''$
is not empty, then its isolated points are the varieties that were limit points in
$\Omega'$
for the sequences containing only isolated points of $\Omega'$, etc.

The structures of a set of non-nilpotent subvarieties for various varieties of  nearly associative metabelian algebras
were studied in~\cite{Pchelintsev81}--\cite{Platonova04}.
A.~V.~Badeev~\cite{Badeev02} provided a chain
$\mathcal V_1\subset\cdots\subset\mathcal V_n\subset\cdots\subset\mathcal V$
of varieties of commutative alternative nil-algebras over a field of characteristic three
such that
$\mathrm{r_t}\left(\mathcal V_n\right)$ is a linear function on~$n$
and
$\mathrm{r_t}\left(\mathcal V\right)=\aleph_0$.
In 2007, S.~V.~Pchelintsev~\cite{Pchelintsev07} constructed a variety $\mathfrak M$ of right alternative metabelian algebras
of almost finite topological rank,
i.~e. a variety $\mathfrak M$ such that
$\mathrm{r_t}\left(\mathfrak M\right)=\aleph_0$ and
$\mathrm{r_t}\left(\mathfrak M'\right)$ is finite
for every proper subvariety
$\mathfrak M'\subset\mathfrak M$.

\subsection*{Formulation of the result}%
Let
$F$ be a \textit{field of characteristic
$\mathrm{char}(F)\neq 2,3$}
and
$\mathrm{RA}_2$
be the
\textit{variety of right alternative metabelian algebras over $F$}
defined by the identities
\begin{align}
\ass{x}{y}{z}+\ass{x}{z}{y}&=0\quad(\textit{the right alternative identity})\label{linprav},\\
\dva{x}{y}\dva{z}{t}&=0\quad(\textit{the metabelian identity})\label{metab},
\end{align}
where
$\ass{x}{y}{z}=\dva{x}{y}z-x\dva{y}{z}$
is the associator of the variables $x,y,z$.
By
$\rl{s}$ we denote the
\textit{subvariety of
$\mathrm{RA}_2$
distinguished by the identity}
\begin{equation}\label{eq Lie-nilp}
\Bigl[\bigl[\ldots\com{x_1}{x_2},\ldots,x_{s}\bigr],x_{s+1}\Bigr]=0
\end{equation}
\textit{of Lie-nilpotency of step~$s$},
where
$\left[x,y\right]=xy-yx$
is the commutator of~$x,y$.

The Specht property of $\rl{s}$
is proved by the author in~\cite{Kuz'min06}.
By virtue of nilpotency of every commutative subvariety of
$\mathrm{RA}_2$, we have
$\mathrm{r_t}\bigl(\rl{1}\bigr)=1$.
S.~V.~Pchelintsev established in~\cite{Pchelintsev81} that
$\mathrm{r_t}\bigl(\rl{2}\bigr)=2$.
In the present paper, we prove the following

\begin{trivlist}
\item[\textbf{ Theorem.}]
\textit{The topological rank of the variety~$\rl{s}$ is equal to $s$ for all
natural~$s$.}
\end{trivlist}

The paper is organized as follows.
In Sec.~\ref{Sec:BasicOperatorRelations}, we provide some preliminary results about the free
$\mathrm{RA}_2\text{-algebra}$
$\svob{\mathrm{RA}_2}$
on a countable set
$X$
of generators over
$F$.
Sec.~\ref{Sec:RelationsOfTheFreeAlgebra}
is devoted to the studying of relations of the free algebra
$\svob{\rl{s}}$.
In Sec.~\ref{Sec:UpperBoundForTheTopologicalRank},
we construct a system of linear generators for the space of multilinear polynomials in
$\svob{\rl{s}}$
of sufficiently high degree
and obtain the upper bound
$\mathrm{r_t}\bigl(\rl{s}\bigr)\leqslant s$
by estimating the values of topological ranks of some subvarieties in
$\rl{s}$
of special type.
Finally, in
Sec.~\ref{Sec:LowBoundForTheTopologicalRank},
we construct an auxiliary
$\rl{s}\text{-superalgebras}$
and considering the identities of their Grassmann envelopes
obtain the low bound
$\mathrm{r_t}\bigl(\rl{s}\bigr)\geqslant s$.

\section{Preliminaries}%
\label{Sec:BasicOperatorRelations}

Throughout the paper,
$F$ is a field of characteristic
$\mathrm{char}(F)\neq 2,3$;
all vector spaces (algebras, superalgebras) are considered over~$F$;
$X=\{x_1,x_2,\ldots\}$
is a countable set;
$\mathfrak A=\svob{\mathrm{RA}_2}$
is a free $\mathrm{RA}_2\text{-algebra}$
on the set $X$ of generators;
$R_x$ and $L_x$ are, respectively,
the operators of right and left multiplication by the element $x$;
$H_x=R_x-L_x$;
$\mathfrak A^*$ is the associative algebra
generated by all the operators $R_x$ and $L_x$, for $x\in\mathfrak A$, acting on~$\mathfrak A^2$
and by the identical mapping
$\mathrm{id}$;
$\vr A$
is the variety generated by an algebra~$A$.

\smallskip

Recall~\cite{Kuz'min06,Pchelintsev07} that $\mathfrak A^*$ satisfies the relations
\begin{align}
&R^2_x=0,\label{eq: RR}\\
&\left[R_xR_y,L_z\right]=0,\label{eq: RRL}\\
&\left[R_x,L_y\right]=-L_xL_y.\label{eq: RL-LL}
\end{align}

Relations~\eqref{eq: RR},~\eqref{eq: RRL} imply immediately the following
\begin{lemma}\label{lm R2 centr}
The operator~$R_xR_y$
lies in the center of~$\mathfrak A^*$.
\end{lemma}

\smallskip

\begin{proposition}\label{predl predvsoot}
The algebra $\mathfrak A^*$ satisfies the relation
\begin{equation}\label{soot 3R2}
3R_xR_y+H_xH_y=2\left[R_x,H_y\right]+H_xR_y+H_yR_x.
\end{equation}
\end{proposition}

\begin{proof}
Using~\eqref{eq: RL-LL}, we have
\begin{multline*}
H_xH_y=\left(R_x-L_x\right)\left(R_y-L_y\right)=R_xR_y+L_xL_y-L_xR_y-R_xL_y=\\
=R_xR_y-\left[R_x,L_y\right]-L_xR_y-L_yR_x-\left[R_x,L_y\right]=\\
=R_xR_y-2\left[R_x,L_y\right]-L_xR_y-L_yR_x.
\end{multline*}
Combining the obtained relation with~\eqref{eq: RR} and~\eqref{eq: RL-LL},
we get
\begin{multline*}
3R_xR_y+H_xH_y=4R_xR_y-2\left[R_x,L_y\right]-L_xR_y-L_yR_x=\\
=2\left[R_x,R_y\right]-2\left[R_x,L_y\right]+\left(H_x-R_x\right)R_y+\left(H_y-R_y\right)R_x=\\
=2\left[R_x,H_y\right]+H_xR_y+H_yR_x.
\end{multline*}
\end{proof}

In what follows, we use the symbol
$T$
as a common notation for the operator symbols
$R$
and
$H$.
The notation
$w=T_{x}\dots T_{y}$
means that each operator symbol of the word
$w$
can be equal to
$R$
or
$H$
independently.
In the case when all operator symbols in some word are assumed to be equal to each other,
we use the notation
$$
%\Bar{T}_{{i_1},\dots,{i_n}}=
\oprT\left(i_1,\dots,i_n\right)=
\left\{
\begin{aligned}
R_{x_{i_1}}\dots R_{x_{i_n}},&\enskip\text{ if }\:T=R,\\
H_{x_{i_1}}\dots H_{x_{i_n}},&\enskip\text{ if }\:T=H
\end{aligned}
\right.
$$
and set
$\oprT(\varnothing)=\mathrm{id}$.

\begin{lemma}\label{predl espH}
The algebra $\mathfrak A^*$ is spanned by the operators
$$
\oprH\left(i_1,\dots,i_n\right)\oprR\left(j_1,\dots,j_m\right).
$$
\end{lemma}

\begin{proof}
Let $\mathcal I$ be a linear span of all operators
$\oprH\left(i_1,\dots,i_n\right)\oprR\left(j_1,\dots,j_m\right)$.
It suffices to prove the inclusions
$\oprR(k)\mathcal I\subseteq\mathcal I$
and
$\mathcal I \oprH(k)\subseteq\mathcal I$.
Note that~\eqref{soot 3R2} yields
$\oprR(i)\oprH(j)\in\mathcal I$.
Hence the inclusion
$\oprR(k)\mathcal I\subseteq\mathcal I$
can be easily proved by induction on the length of the operator
$\oprH\left(i_1,\dots,i_n\right)$.
At the same time,
Lemma~\ref{lm R2 centr} implies
$\mathcal I \oprH(k)\subseteq\mathcal I$.
\end{proof}

Let
$\mathcal L$
be a \textit{linear span in~$\mathfrak A^*$ of all operators of the form}
$$
L_{x_i}w,\quad w\in\mathfrak A^*.
$$
By virtue of~\eqref{eq: RL-LL},
$\mathcal L$
forms an ideal of~$\mathfrak A^*$
and by induction on~$n$ one can prove the congruence
\begin{equation}\label{eq: HkequivRk}
\oprH\left(1,\dots,n\right)\equiv
\oprR\left(1,\dots,n\right)
\pmod{\mathcal L}, \quad n\in\mathbb N.
\end{equation}

\section{Relations of the free $\rl{s}\text{-algebra}$}%
\label{Sec:RelationsOfTheFreeAlgebra}

Let
${\mathfrak A}_{s}=\svob{\rl{s}}$
be the
\textit{free $\rl{s}\text{-algebra}$ on the set
$X$ of generators.}
Lemma~\ref{predl espH} implies immediately the following

\begin{lemma}\label{lm espH}
The linear span of all operators of degree~$d\geqslant s$ in~${\mathfrak A}_{s}^*$
is spanned by the operators
$$
\oprH\left(i_1,\dots,i_n\right)\oprR\left(j_1,\dots,j_{d-n}\right),\quad n<s.
$$
\end{lemma}

\smallskip

In what follows, the term
\textit{"polynomial"}
means a homogeneous polynomial of degree not less then two.

\begin{definition}
Let~$\approx$ be a symmetric relation on the set of polynomials of $\mathfrak A$
such that
$f_0\approx f_{1}$
if
$f_{i}=f_{1-i}\oprR\left(j_1,\dots,j_{2k}\right)$,
$i\in\{0,1\}$,
and $f_{1-i}$ doesn't depend on the variables
$x_{j_1},\dots, x_{j_{2k}}$.
By the same symbol~$\approx$ we denote the induced relation on~$\mathfrak A^*$:
$\xi\approx\eta$
for
$\xi,\eta\in\mathfrak A^*$
if
$(x_i x_j)\xi\approx(x_i x_j)\eta$
and
$\xi,\eta$ do not depend on $x_i,x_j$.
\end{definition}

\begin{proposition}
The algebra ${\mathfrak A}_{s}$ satisfies the relation
\begin{equation}\label{eq: xkub}
x^3\approx0.
\end{equation}
\end{proposition}
\begin{proof}
Using~\eqref{linprav} and~\eqref{metab}, we have
$$
2yx^3=y\left(\jord{x}{x^2}\right)=\left(yx^2\right)x=
\tri{y}{x}{x}x=\dva{y}{x}x^2=0.
$$
Hence,
$x^3\mathcal L=0$.
Therefore
applying~\eqref{eq: HkequivRk},
for even
$n\geqslant s$,
and taking into account~\eqref{eq Lie-nilp},
we obtain
$$
x^3\approx x^3\!
\oprR\left(1,\dots,n\right)=x^3\!\oprH\left(1,\dots,n\right)=0.
$$
\end{proof}

We say that
\textit{almost all polynomials of
${\mathfrak A}_{s}$
(operators of ${\mathfrak A}_{s}^*$)
satisfy some condition $\vartheta$}
if there is a natural $n$ such that
$\vartheta$
holds for all polynomials
(operators) of degree more then $n$.

\begin{lemma}\label{lm approx}
If
$f\approx 0$
for
$f\in{\mathfrak A}_{s}$,
then almost all operators of
${\mathfrak A}_{s}^*$
annihilate $f$.
\end{lemma}

\begin{proof}
Assume that
$f\oprR\left(j_1,\dots,j_{2k}\right)=0$,
where
$f$ doesn't depend on
$x_{j_1},\dots, x_{j_{2k}}$.
In view of Lemma~\ref{lm espH},
every operator word
$\xi\in{\mathfrak A}_{s}^*$ of the degree~$d\geqslant s+2k$
can be represented as
$$
\xi=\eta\oprR\left(j_1,\dots,j_{2k}\right), \quad \eta\in{\mathfrak A}_{s}^*.
$$
Hence by Lemma~\ref{lm R2 centr}, we have
$$
f\xi=
f\oprR\left(j_1,\dots,j_{2k}\right)\eta=0.
$$
\end{proof}

\begin{lemma}\label{lm opercosn}
Almost all operators of
${\mathfrak A}_{s}^*$
are skew-symmetric with respect to all their variables.
\end{lemma}

\begin{proof}
We set $w\in{\mathfrak A}_{s}^2$.
By virtue of~\eqref{metab} and~\eqref{eq: RR},
the partial linearization~(see~\cite[Chap.~1]{Zhevlakov-Slin'ko-Shestakov-Shirshov}) of~\eqref{eq: xkub} has the form
$$
\dva{w}{x}x+\dva{x}{w}x+x^2w=\dva{x}{w}x\approx0,
$$
whence,
$H_{x}R_{x}=-L_{x}R_{x}\approx0$.
Thus in view of Lemma~\ref{lm espH}, it remains to calculate
$$
H_{x}H_{x}\approx H_{x}H_{x}R_yR_z\approx -H_{x}H_{y}R_xR_z =
-H_{x}R_xR_zH_{y}\approx0.
$$
\end{proof}

\begin{proposition}
The algebra
${\mathfrak A}_{s}$
satisfies the relation
\begin{equation}\label{eq: xyTxTy}
(xy)T_xT_y\approx0.
\end{equation}
\end{proposition}

\begin{proof}
By virtue of Lemmas~\ref{lm R2 centr},~\ref{lm opercosn},
it suffices to verify that
$(xy)R_xR_y\approx0$.
Using~\eqref{eq: RR},~\eqref{linprav}, and Lemma~\ref{lm opercosn}, we have
$$
\dva{x}{y}R_xR_y=-\dva{x}{y}R_yR_x=y^2L_xR_x\approx0.
$$
\end{proof}

\begin{proposition}
\label{predl soot RH2}
The algebra ${\mathfrak A}_{s}^*$ satisfies the relations
\begin{align}
3R_xR_y-2\left[R_x,H_y\right]+H_xH_y&\approx0,\label{soot 3R^2Rn}\\
\left[R_x,H_yH_z\right]&\approx0.\label{soot RH2}
\end{align}
\end{proposition}

\begin{proof}
By Lemma~\ref{lm opercosn},
relation~\eqref{soot 3R^2Rn} follows from~\eqref{soot 3R2}.
Using~\eqref{soot 3R^2Rn} and combining Lemmas~\ref{lm R2 centr},~\ref{lm opercosn} with
the Jacobian identity,
we obtain
\begin{multline*}
\left[R_x,H_yH_z\right]\approx 2\bigl[R_x,\left[R_y,H_z\right]\bigr]\approx\\
\approx\bigl[R_x,\left[R_y,H_z\right]\bigr]-\bigl[R_y,\left[R_x,H_z\right]\bigr]
=\bigl[H_z,\left[R_y,R_x\right]\bigr]=0.
\end{multline*}
\end{proof}

\begin{definition}
Let
$\mathcal I$
be an ideal of
${\mathfrak A}_{s}^*$.
For
$\xi,\eta\in{\mathfrak A}_{s}^*$
we write
$\xi\cong \eta \pmod{\mathcal I}$
if there is a
$\theta\in \mathcal I$
such that
$\xi-\eta\approx \theta$.
\end{definition}

\smallskip

Let
$\mathcal H_n$ $(n<s)$
be the
\textit{ideal of~$\mathfrak A_{s}^*$ generated by all the elements}
$\oprH\left(i_1,\dots,i_{n}\right)$.

\begin{proposition}
The algebra ${\mathfrak A}_{s}^*$ satisfies the relation
\begin{equation}\label{eq reduction to 2n}
\oprH\left(1,\dots,2t\right)\cong 0\pmod{\mathcal H_{2t+1}}.
\end{equation}
\end{proposition}

\begin{proof}
We set
$\eta=\oprH\left(1,\dots,2t\right)$.
Applying~\eqref{soot 3R^2Rn} and~\eqref{soot RH2},
we have
$$
3\eta\approx3\eta R_xR_y\cong
2\eta R_x H_y
\approx
2R_x\eta H_y\cong0
\pmod{\mathcal H_{2t+1}}.
$$
\end{proof}

\section{Upper bound for the topological rank of~$\rl{s}$}%
\label{Sec:UpperBoundForTheTopologicalRank}

\begin{definition}
An
\textit{$n\text{-allotted}$ variety}
($1\leqslant n\leqslant s$)
is a subvariety
$\mathcal V$ of $\rl{s}$
such that the free
$\mathcal V\text{-algebra}$
on the set~$X$ of generators satisfies the relation
\begin{equation}\label{rel n-allot}
\varphi\left(x_1,\dots,x_{n+1}\right)\approx0,
\end{equation}
where
$$
\varphi\left(x_1,\dots,x_{n+1}\right)
=\left\{
\begin{aligned}
\Bigl[\bigl[\ldots\com{x_1}{x_2},\ldots,x_{n}\bigr],x_{n+1}\Bigr],
&\enskip\text{ if $n$ is even},\\
\Bigl[\bigl[\ldots\com{x_1x_2}{x_3},\ldots,x_{n}\bigr],x_{n+1}\Bigr],
&\enskip\text{ if $n$ is odd}.
\end{aligned}\right.
$$
\end{definition}
By definition, every
$1\text{-allotted}$
variety
$\mathfrak M$
is right nilpotent.
Moreover, applying Lemma~\ref{lm espH}, it is not hard to prove that
$\mathfrak M$ is nilpotent and, consequently,
$\mathrm{r_t}\left(\mathfrak M\right)=1$.
We also stress that the variety
$\rl{s}$
is
$s\text{-allotted}$:
for even $s$, it is clear by definition and,
for odd $s$, it follows from~\eqref{eq reduction to 2n}.

\medskip

Let
$\mathcal V$
be an
\textit{$n\text{-allotted}$ variety}
$(n\geqslant2)$,
$\mathcal A$
be the
\textit{free
$\mathcal V\text{-algebra}$ on the set $X$ of generators},
and
${\mathcal P}_{d,n}$
 $(d\geqslant3)$
be the
\textit{subspace of multilinear polynomials in
$\mathcal A$
on the variables
$x_1,\dots,x_d$.}
In order to avoid complicated formulas while writing down
the polynomials of
${\mathcal P}_{d,n}$
we omit the indices of variables
at the operator symbols
and assume them to be arranged at the ascending order.
For example, the notation
$w=(x_2x_5)H^2R^3$
means the monomial
$$
w=(x_2x_5)\oprH\left(1,3\right)\oprR\left(4,6,7\right).
$$

\begin{definition}
{\upshape
\textit{Regular words} are the polynomials of
${\mathcal P}_{d,n}$
of the following types:
\begin{align*}
1)&\enskip\,\jos{x_1}{x_i}H^{2j}R^{d-2j-2},\\
2)&\enskip\,\com{x_1}{x_i}H^{2j}R^{d-2j-2},\\
3)&\enskip\,\com{x_2}{x_3}H^{2j}R^{d-2j-2},\\
4)&\enskip\,\com{x_1}{x_2}H^{2k-1}R^{d-2k-1},
\end{align*}
where
$i=2,3,\dots,d;\,$
$j=0,1,\dots,t-1;\,$
$k=1,2,\dots, n-t-1;\,$
$t=\slch{n}{2}$.}
\end{definition}

\begin{lemma}\label{lm pravl slova}
Almost all polynomials of
$\bigcup_{d=3}^{\infty}{\mathcal P}_{d,n}$
are linear combinations of regular words.
\end{lemma}

\begin{proof}
By Lemma~\ref{lm opercosn},
there is a degree
$d$
such that every monomial
$$
(x_1x_2)T_{3}\dots T_{d}\in{\mathcal P}_{d,n}
$$
is skew-symmetric w.r.t.
$x_3,\dots,x_d$.
Consequently, in view of Lemma~\ref{lm espH} and relation~\eqref{eq reduction to 2n},
${\mathcal P}_{d,n}$
can be spanned by the polynomials
$$
\jos{x_i}{x_j}H^{k}R^{d-k-2},\quad
\com{x_i}{x_j}H^{k}R^{d-k-2},
$$
where
$a\circ b=ab+ba$,
$1\leqslant i<j\leqslant d$,
and
$k=0,1,\dots,2t-1$.

Linearizing~\eqref{eq: xkub} and~\eqref{eq: xyTxTy}, we have
\begin{align*}
\jos{x}{y}T_z+\jos{y}{z}T_x+\jos{z}{x}T_y&\approx0,\\
\com{x}{y}T_zT_t+\com{x}{t}T_zT_y+\com{z}{y}T_xT_t+\com{z}{t}T_xT_y&\approx0.
\end{align*}
Applying these relations,
it is not hard to prove that
${\mathcal P}_{d,n}$
can be spanned by the polynomials:
\begin{align*}
1')&\enskip\jos{x_1}{x_i}H^{k}R^{d-k-2},\\
2')&\enskip\com{x_1}{x_i}H^{k}R^{d-k-2},\\
3')&\enskip\com{x_2}{x_3}H^{k}R^{d-k-2},
\end{align*}
where
$i=2,\dots,d$
and
$k=0,1,\dots,2t-1$.

By $\mathcal W$ denote the linear span of all regular words of types~$1)$--$3)$.
Note that the polynomials of types~$1')$--$3')$ lie in~$\mathcal W$ for even~$k$.
Let us verify that the polynomials of types~$1')$--$3')$ for odd~$k$
can be represented as linear combinations of regular words.

By virtue of~\eqref{linprav},
we have
$$
\jos{x}{y}H_z=\jos{x}{y}z-z\jos{x}{y}
=\jos{x}{y}z-\dva{z}{x}y-\dva{z}{y}x.
$$
Hence, in view of~\eqref{soot RH2},
every polynomial of type~$1')$ lie in~$\mathcal W$.
Further, using
Lemmas~\ref{lm R2 centr},~\ref{lm opercosn},
the partial linearization
$$
\dva{x}{y}y+\dva{y}{x}y+y^2x\approx0
$$
of~\eqref{eq: xkub},
identity~\eqref{linprav} and relation~\eqref{soot 3R^2Rn},
we get
\begin{multline*}
\com{x}{y}H_y\approx\com{x}{y}H_yR_zR_u
\approx\com{x}{y}R_zR_uH_y=\com{x}{y}R_yR_zH_u=\\
=\bigl(\dva{x}{y}y-\dva{y}{x}y\bigr)R_zH_u
\approx\bigl(2\dva{x}{y}y+y^2x\bigr)R_zH_u
=y^2\left(2L_x+R_x\right)R_zH_u=\\
=y^2\left(3R_x-2H_x\right)R_zH_u
\approx y^2\left(2\com{R_x}{H_z}-H_xH_z-2H_xR_z\right)H_u\approx\\
\approx y^2\left(2R_x-H_x\right)H_zH_u
= y^2\left(R_x+L_x\right)H_zH_u
=\left(y^2R_x+\dva{x}{y}R_y\right)H_zH_u.
\end{multline*}
In view of~\eqref{soot RH2} the obtained relation implies that
the polynomials of types~$2'),3')$ for odd~$k$
are skew-symmetric modulo~$\mathcal W$ with respect to all their variables.
Therefore every polynomial of type~$2'),3')$
is proportional  modulo~$\mathcal W$ to a regular word of type~$4)$.
\end{proof}

\begin{lemma}\label{lm 2n iz 2n+1}
For every
$n\text{-allotted}$
variety
$\mathcal V$
$(n\geqslant2)$
there is a punctured neighborhood
$\textstyle\stackrel{\circ}{\mathrm U}_d(\mathcal V)$
such that every variety of
$\textstyle\stackrel{\circ}{\mathrm U}_d(\mathcal V)$
is
$(n-1)\text{-allotted}$.
\end{lemma}

\begin{proof}
By virtue of the restriction
$\mathrm{char}(F)\neq 2$,
Lemma~\ref{lm opercosn} and relation~\eqref{eq: xkub} imply that in some punctured neighborhood of
$\mathcal V$
every variety can be defined by a system of identities
where all polynomials starting from some sufficiently high degree are multilinear.
Consequently by Lemma~\ref{lm pravl slova}, we can choose a punctured neighborhood
$\textstyle\stackrel{\circ}{\mathrm U}_d(\mathcal V)$
such that every variety
$\mathfrak M\in\textstyle\stackrel{\circ}{\mathrm U}_d(\mathcal V)$
satisfies an identity
$f=0$,
where
$f$
is a nontrivial linear combination of regular words of
${\mathcal P}_{d,n}$.
Let
$\mathcal A$
be the free
$\mathfrak M\text{-algebra}$ on the set $X$ of generators.
We write down relation~\eqref{rel n-allot} shortly as
$\mathcal A^2 H^{2t}\approx0$
if
$n=2t+1$
and as
$\mathcal A H^{2t}\approx0$
if
$n=2t$.

First consider the case $n=2t+1$.
We prove that relation
$\mathcal A^2 H^{2t}\approx0$
and identity
$f=0$
imply
$\mathcal A H^{2t}\approx0$.
By Lemma~\ref{lm pravl slova}, $f$ can be presented in the form
$$
f\equiv\sum_{i=2}^d\sum_{j=0}^{t-1}
\Bigl(\alpha^{(i)}_{2j}\jos{x_1}{x_i}H^{2j}R^{d-2j-2}
+\alpha^{(i)}_{2j+1}\com{x_1}{x_i}H^{2j}R^{d-2j-2}
\Bigr)\pmod{\mathcal W_{3,4}},
$$
where
$\mathcal W_{3,4}$
is the linear span of regular words of types~$3),4)$.
We fix
$i\geqslant4$
and a minimal index
$\ell$
such that
$\alpha^{(i)}_\ell\neq0$.
Then by the substitution
$x_i:=aH^{2t-\ell}$,
for
$a\in\mathcal A$,
using the equality
$R_y+L_y=2R_y-H_y$ and relation~\eqref{soot RH2}, we obtain
$aH^{2t}\approx0$.
Otherwise, we can rewrite $f$ in the form
$$
f=\sum_{k=0}^{t-1}g_k+\sum_{k=1}^{t}h_k,
$$
where
\begin{align*}
g_0&=\Bigl(\alpha_0\com{x_1}{x_2}x_3
+\beta_0\com{x_3}{x_1}x_2
+\gamma_0\com{x_2}{x_3}x_1\Bigr)R^{d-3},\\
h_t&=
\zeta_t\com{x_1}{x_2}H^{2t-1}R^{d-2k-1},\\
\intertext{and}
g_k&=
\Bigl(
\alpha_{k}\bigl[\com{x_1}{x_2},x_3\bigr]
+\beta_{k}\bigl[\com{x_3}{x_1},x_2\bigr]
+\gamma_{k}\bigl[\com{x_2}{x_3},x_1\bigr]\Bigr)H^{2k-1}R^{d-2k-2},\\
h_k&=
\delta_{k}\jos{x_1}{x_2}H^{2k}R^{d-2k-2}
+\varepsilon_{k}\jos{x_1}{x_3}H^{2k}R^{d-2k-2}+
\zeta_k\com{x_1}{x_2}H^{2k-1}R^{d-2k-1},
\end{align*}
for $k=1,\dots,t-1$.
If at least one of the coefficients $\alpha_0,\beta_0,\gamma_0$
is not zero, then by three successive substitutions
$x_i:=aH^{2t-1}$ $(i=1,2,3)$,
we have
$$
\left\{
\begin{aligned}
\left(\alpha_0+\beta_0\right)aH^{2t}&\approx0,\\
\left(\alpha_0+\gamma_0\right)aH^{2t}&\approx0,\\
\left(\beta_0+\gamma_0\right)aH^{2t}&\approx0.
\end{aligned}\right.
$$
Hence
in view of the restriction
$\mathrm{char}(F)\neq 2$,
we obtain either
$aH^{2t}\approx0$
or~$g_0=0$.
Further, if
$\varepsilon_1\neq0$,
then by the substitution
$x_3:=aH^{2t-2}$,
we have $aH^{2t}\approx0$.
Otherwise, if at least one of the coefficients
$\delta_1$ or $\zeta_1$ is not zero, by two successive substitutions
$x_i:=aH^{2t-2}$ $(i=1,2)$, we obtain
$$
\left\{
\begin{aligned}
\left(2\delta_1+\zeta_1\right)aH^{2t}&\approx0,\\
\left(2\delta_1-\zeta_1\right)aH^{2t}&\approx0.
\end{aligned}\right.
$$
Thus we have either
$aH^{2t}\approx0$
or
$h_1=0$ and, consequently,
$f$
gets the form
$$
f=\sum_{k=1}^{t-1}g_k+\sum_{k=2}^{t}h_k.
$$
Therefore by the same arguments as above,
we obtain either
$aH^{2t}\approx0$ or
$$
g_1=h_2=\dots=g_{t-2}=h_{t-1}=g_{t-1}=0
$$
and
$$
f=h_t=\zeta_{t}\com{x_1}{x_2}H^{2t-1}R^{d-2t-1}.
$$
But in this case, the assumption of the lemma implies
$\zeta_{t}\neq0$
and, consequently,
$\mathcal A H^{2t}\approx0$.

Now consider the case
$n=2t$.
We need to prove that relation
$\mathcal A H^{2t}\approx0$
and identity
$f=0$
imply
$\mathcal A^2 H^{2t-2}\approx0$.
Unlike the case
$n=2t+1$,
the regular word of type~$4)$ corresponding to the index~$k=t$
vanishes.
All the other regular words are the same.
Therefore by the similar arguments as above,
reducing per unit the power
$p(t)$
for every substitution
$x_i:=aH^{p(t)}$
and assuming
$a\in\mathcal A^2$,
one can prove that
$\mathcal A^2 H^{2t-1}\approx0$.
By~\eqref{eq reduction to 2n}, the obtained relation yields
$\mathcal A^2 H^{2t-2}\approx0$.
\end{proof}

As it was stressed above, every
$1\text{-allotted}$
variety has the topological rank~$1$.
Consequently
Lemma~\ref{lm 2n iz 2n+1} implies that the topological rank of every
$n\text{-allotted}$ variety is not more than~$n$.
In particular,
$\mathrm{r_t}\bigl(\rl{s}\bigr)\leqslant s$.

\section{Low bound for the topological rank of~$\rl{s}$}%
\label{Sec:LowBoundForTheTopologicalRank}

Let
$\mathcal A=\mathcal A_0\oplus\mathcal A_1$
be a \textit{superalgebra}
($\mathbb Z_2\text{-graded algebra}$) with the
\textit{even part $\mathcal A_0$}
and the
\textit{odd part $\mathcal A_1$}, i.~e.
$\mathcal A_i\mathcal A_j\subseteq\mathcal A_{i+j\;\,(\mathrm{mod} 2)}$
for
$i,j\in\{0,1\};$
${\mathrm G}$
be the \textit{Grassmann algebra} on a countable set of anticommuting generators
$\left\{e_1,e_2,\ldots\mid e_ie_j=-e_je_i\right\}$
with the natural $\mathbb Z_2\text{-grading}$
($\mathrm G_0$ and $\mathrm G_1$ are spanned by the words of even and, respectively, odd length on $\left\{e_i\right\}$).
The \textit{Grassmann envelope}
${\mathrm G}\left(\mathcal A\right)$
of
$\mathcal A$ is the subalgebra
$
{\mathrm G_0}\otimes \mathcal A_0+{\mathrm G_1}\otimes \mathcal A_1
$
of the tensor product
${\mathrm G}\otimes \mathcal A$.
It is well known that
${\mathrm G}\left(\mathcal A\right)$
satisfies a multilinear identity
$f=0$
iff
$\mathcal A$
satisfies the certain graded identity
$\Tilde{f}=0$
called the
\textit{superization of~$f=0$.}
Here
$\Tilde{f}$ denotes the so-called
\textit{superpolynomial corresponding to $f$} and
we also say that
$\mathcal A$ satisfies the \textit{superidentity}
$\Tilde{f}=0$.
The descriptions of the process of constructing of superpolynomials
(the \textit{superizing process})
can be found in~\cite{Zelmanov-Shestakov90}--\cite{Vaughan-Lee98}.
For a given variety
$\mathcal V$
of algebras,
$\mathcal A$ is said to be a
$\mathcal V\textit{-superalgebra}$ if
${\mathrm G}\left(\mathcal A\right)\in\mathcal V$,
i.~e. if
$\mathcal A$
satisfies all the superizations of the defining identities of~$\mathcal V$.

\smallskip

Let
$\varepsilon$
be one of the elements
$0,1\in F$
and
${\mathcal A}^{(\varepsilon)}
={\mathcal A}^{(\varepsilon)}_0\oplus {\mathcal A}^{(\varepsilon)}_1
%=F\cdot x\bigoplus_{i=0}^{\infty}\bigoplus_{j=0}^{\infty} F\cdot a_{i,j}
$
be a superalgebra with the countable basis
$x$, $a_{i,j}$ $(i,j=0,1,\dots)$
such that
$x\in{\mathcal A}^{(\varepsilon)}_1$
and
$a_{i,j}\in{\mathcal A}^{(\varepsilon)}_0$
iff
$i\equiv j\pmod{2}$.
We introduce the multiplication of the basis elements of
${\mathcal A}^{(\varepsilon)}$
as follows:
\begin{gather*}
x^2=\frac{\varepsilon}{2} a_{0,0},\quad
a_{i,j}\cdot x=a_{i,\,j+1},\quad
x\cdot a_{i,\,2j}={(-1)}^{i}\left(a_{i,\,2j+1}-a_{i+1,\,2j}\right),\\
x\cdot a_{i,\,2j+1}=\frac{{(-1)}^{i}}{2}\left(a_{i,\,2j+2}-2a_{i+1,\,2j+1}+a_{i+2,\,2j}\right),
\end{gather*}
and all the other products are zero.
By definition, it is not hard to see that
${\mathcal A}^{(\varepsilon)}$
is a metabelian algebra.

\begin{lemma}\label{predl GApravH}
The algebra
${\mathcal A}^{(\varepsilon)}$ is an $\mathrm{RA}_2\text{-superalgebra}$.
\end{lemma}

\begin{proof}
We check that
${\mathcal A}^{(\varepsilon)}$
satisfies  the superization of~\eqref{linprav}:
$$
\ass{a}{b}{c}+{\left(-1\right)}^{\md{b}\md{c}}\ass{a}{c}{b}=0,
$$
where
$a,b,c$
are homogeneous basis elements of
${\mathcal A}^{(\varepsilon)}$
and
$\left|a\right|$
denotes the parity of~$a$,
i.~e.
$\left|a\right|=k$
for
$a\in{\mathcal A}^{(\varepsilon)}_k\;(k=0,1)$.
In view of metability of
${\mathcal A}^{(\varepsilon)}$
and the odd parity of~$x$,
it suffices to verify the relation
\begin{equation}\label{soot provLR L2}
a_{i,j}\bigl(\com{L_x}{R_x}-{(-1)}^{i+j}L^2_x\bigr)=0.
\end{equation}
First for even $j$, we calculate
\begin{multline*}
a_{i,j}\com{L_x}{R_x}
={(-1)}^{i}\left(
a_{i,\,j+1}-a_{i+1,\,j}\right)R_x-a_{i,\,j+1}L_x=\\
={(-1)}^{i}\left(a_{i,\,j+2}-a_{i+1,\,j+1}\right)
-\frac{{(-1)}^{i}}{2}\left(a_{i,\,j+2}-2a_{i+1,\,j+1}+a_{i+2,\,j}\right)=\\
\shoveright{=\frac{{(-1)}^{i}}{2}\left(
a_{i,\,j+2}-a_{i+2,\,j}
\right);}\\
\shoveleft{a_{i,j}\,L^2_x
={(-1)}^{i}\left(a_{i,\,j+1}-a_{i+1,\,j}\right)L_x=}\\
=\frac{1}{2}\left(a_{i,\,j+2}-2a_{i+1,\,j+1}+a_{i+2,\,j}\right)+a_{i+1,\,j+1}-a_{i+2,\,j}=\\
=\frac{1}{2}\left(
a_{i,\,j+2}-a_{i+2,\,j}
\right).
\end{multline*}
To conclude the proof it remains to make the similar calculations for odd~$j$:
\begin{multline*}
a_{i,j}\com{L_x}{R_x}
=\frac{{(-1)}^{i}}{2}\left(a_{i,\,j+1}-2a_{i+1,\,j}+a_{i+2,\,j-1}\right)R_x-a_{i,\,j+1}L_x=\\
=\frac{{(-1)}^{i}}{2}\left(a_{i,\,j+2}-2a_{i+1,\,j+1}+a_{i+2,\,j}\right)
+{(-1)}^{i+1}\left(a_{i,\,j+2}-a_{i+1,\,j+1}\right)=\\
=\frac{{(-1)}^{i+1}}{2}\left(a_{i,\,j+2}-a_{i+2,\,j}\right);
\end{multline*}
\begin{multline*}
a_{i,j}\,L^2_x
=\frac{{(-1)}^{i}}{2}\left(a_{i,\,j+1}-2a_{i+1,\,j}+a_{i+2,\,j-1}\right)L_x=\\
=\frac{1}{2}\left(
a_{i,\,j+2}-a_{i+1,\,j+1}
+a_{i+1,\,j+1}
-2a_{i+2,\,j}+a_{i+3,\,j-1}+a_{i+2,\,j}-a_{i+3,\,j-1}\right)=\\
=\frac{1}{2}\left(a_{i,\,j+2}-a_{i+2,\,j}\right).
\end{multline*}
\end{proof}

\begin{proposition}\label{predl supercom}
The algebra ${\mathcal A}^{(\varepsilon)}$ satisfies the relations
\begin{align}
\label{soot supercom}
{\left[a_{i,\,2j},x\right]}_{\mathrm s}&=a_{i+1,2j},\\
\label{soot supercom2}
{\bigl[{\left[a_{i,j},x\right]}_{\mathrm s},x\bigr]}_{\mathrm s}&=a_{i+2,\,j},
\end{align}
where ${\left[a,b\right]}_{\mathrm s}=ab-{(-1)}^{\md{a}\md{b}}ba$ is a supercommutator
of the elements~$a,b$.
\end{proposition}

\begin{proof}
First we calculate
$$
{\left[a_{i,\,2j},x\right]}_{\mathrm s}
=a_{i,\,2j}\,R_x-{(-1)}^{i} a_{i,\,2j}\,L_x
=a_{i,\,2j+1}-\left(a_{i,\,2j+1}-a_{i+1,\,2j}\right)=a_{i+1,\,2j}.
$$
Thus~\eqref{soot supercom} and, consequently,~\eqref{soot supercom2}, for even~$j$, are proved.
Further, for odd~$j$, we have
\begin{multline*}
{\left[a_{i,j},x\right]}_{\mathrm s}
=a_{i,j}\,R_x+{(-1)}^{i} a_{i,j}\,L_x
=a_{i,\,j+1}+\frac{1}{2}\left(a_{i,\,j+1}-2a_{i+1,\,j}+a_{i+2,\,j-1}\right)=\\
=\frac{3}{2}a_{i,\,j+1}-a_{i+1,\,j}+\frac{1}{2}a_{i+2,\,j-1}.
\end{multline*}
Finally, combining the obtained relation with~\eqref{soot supercom}, we get
\begin{multline*}
{\bigl[{\left[a_{i,j},x\right]}_{\mathrm s},x\bigr]}_{\mathrm s}
=\frac{1}{2}{\bigl[\left(3a_{i,\,j+1}-2a_{i+1,\,j}+a_{i+2,\,j-1}\right),x\bigr]}_{\mathrm s}=\\
=\frac{1}{2}\left(3a_{i+1,\,j+1}-3a_{i+1,\,j+1}+2a_{i+2,\,j}-a_{i+3,\,j-1}+a_{i+3,\,j-1}\right)=a_{i+2,\,j}.
\end{multline*}
\end{proof}

\medskip

Let
${\mathcal I}^{(k)}$
$(k\in\mathbb N)$
be the span of all elements
$a_{i,j}\in{\mathcal A}^{(\varepsilon)}$
such that
$i\geqslant 2k$.
By the definition of multiplication in
${\mathcal A}^{(\varepsilon)}$,
every nonzero product
$a_{i,j}T_x$
is a linear combination of elements
$a_{i',j'}$
such that
$i'\geqslant i$.
Consequently,
${\mathcal I}^{(k)}$
is an ideal of~${\mathcal A}^{(\varepsilon)}$.
For every natural
$n\geqslant 2$,
we introduce the quotient superalgebra
${\mathcal A}^{\left<n\right>}$
as follows:
$$
{\mathcal A}^{\left<2k\right>}=\factor{{\mathcal A}^{(0)}}{{\mathcal I}^{(k)}},\qquad
{\mathcal A}^{\left<2k+1\right>}=\factor{{\mathcal A}^{(1)}}{{\mathcal I}^{(k)}}.
$$

\begin{lemma}\label{lm vGobAN}
The algebra
${\mathcal A}^{\left<n\right>}$ is an $\rl{n}\text{-superalgebra}$.
\end{lemma}
\begin{proof}
Taking into account Lemma~\ref{predl GApravH}, it suffices to prove that
${\mathcal A}^{\left<n\right>}$
satisfies the superization of~\eqref{eq Lie-nilp}.
By virtue of metability of
${\mathcal A}^{\left<n\right>}$,
we need to verify the relation
$$
\underbrace{\bigl[\ldots\bigl[[}_{n}{\mathcal A}^{\left<n\right>},{x]}_{\mathrm s},x\bigr]_{\mathrm s},\ldots,x\bigr]_{\mathrm s}=0.
$$
By the definition of ${\mathcal A}^{\left<n\right>}$,
using~\eqref{soot supercom2},
for
$k=\slch{n}{2}$,
we have
$$
\underbrace{\bigl[\ldots\bigl[[}_{2k }a_{i,j},{x]}_{\mathrm s},x\bigr]_{\mathrm s},\ldots,x\bigr]_{\mathrm s}
=a_{i+2k,\,j}=0.
$$
Thus the required relation is proved for even~$n$.
In the case of odd~$n$,
it remains to check the following:
$$
\underbrace{\bigl[\ldots\bigl[[}_{2k+1}x,{x]}_{\mathrm s},x\bigr]_{\mathrm s},\ldots,x\bigr]_{\mathrm s}
=\underbrace{\bigl[\ldots\bigl[[}_{2k}a_{0,0},{x]}_{\mathrm s},x\bigr]_{\mathrm s},\ldots,x\bigr]_{\mathrm s}
=a_{i+2k,\,j}=0.
$$
\end{proof}

\begin{lemma}\label{lm vGobAN-1}
The variety
$\vr\Gob{{\mathcal A}^{\left<n\right>}}$
is not $(n-1)\text{-allotted}$.
\end{lemma}

\begin{proof}
In the case
$n=2k+1$, we need to verify that
$$
\underbrace{\bigl[\ldots\bigl[[}_{2k}{\mathcal A}^{\left<n\right>},{x]}_{\mathrm s},x\bigr]_{\mathrm s},\ldots,x\bigr]_{\mathrm s}R^j_x\neq0,
\quad j\in\mathbb N.
$$
Applying~\eqref{soot supercom}, we calculate
$$
\underbrace{\bigl[\ldots\bigl[[}_{2k}x,{x]}_{\mathrm s},x\bigr]_{\mathrm s},\ldots,x\bigr]_{\mathrm s}R^j_x
=\underbrace{\bigl[\ldots\bigl[[}_{2k-1}a_{0,0},{x]}_{\mathrm s},x\bigr]_{\mathrm s},\ldots,x\bigr]_{\mathrm s}R^j_x
=a_{2k-1,j}\neq0.
$$

For
$n=2k$,
we need to check that
$$
\underbrace{\bigl[\ldots\bigl[[}_{2k-2}\left({\mathcal A}^{\left<n\right>}\right)^2,{x]}_{\mathrm s},x\bigr]_{\mathrm s},\ldots,x\bigr]_{\mathrm s}R^j_x\neq0,
\quad j\in\mathbb N.
$$
Using~\eqref{soot supercom2}, we get
$$
\underbrace{\bigl[\ldots\bigl[[}_{2k-2}a_{0,1},{x]}_{\mathrm s},x\bigr]_{\mathrm s},\ldots,x\bigr]_{\mathrm s}R^j_x=a_{2k-2,\,j+1}\neq0.
$$
\end{proof}

\smallskip

In view of~\eqref{eq reduction to 2n}, Lemma~\ref{lm vGobAN} implies that the variety
$\vr\Gob{{\mathcal A}^{\left<n\right>}}$ is
$n\text{-allotted}$.
Consequently by Lemma~\ref{lm vGobAN-1}, we have
$\mathrm{r_t}\bigl(\rl{n}\bigr)\geqslant n$
for $n=2,\dots,s$.
Finally, comparing this estimate with the result of Sec.~\ref{Sec:UpperBoundForTheTopologicalRank},
we obtain
$\mathrm{r_t}\bigl(\rl{s}\bigr)=s$.

\section*{Acknowledgments}%
This article was carried out at the Department of Mathematics and Statistics (IME) of the University of S\~ao Paulo
as a part of the author's post-doc project supported by the S\~ao Paulo Research Foundation (FAPESP), grant 2010/51880--2.
The author is very thankful to the IME for the kind hospitality and the creative atmosphere,
to his supervisor Prof. I.~P.~Shestakov for his attention to this article,
and to Prof. S.~V.~Pchelintsev
for suggesting the problem and for the useful discussions on the obtained results.

\end{document}